\documentclass[12pt, a4paper]{amsart}
\usepackage[english]{babel}
\usepackage{amsmath,amsfonts,amssymb,amsthm,graphics,epsfig}
\usepackage{hyperref}

\usepackage[margin=2.5cm]{geometry}
\usepackage{epstopdf}

\newtheorem{thm}{Theorem}[section]
\newtheorem{lem}[thm]{Lemma}
\newtheorem{cor}[thm]{Corollary}

\theoremstyle{definition}
\newtheorem{defin}[thm]{Definition}
\newtheorem*{rem}{Remark}

\newcommand { \ib }[1] {\textit{\textbf{#1}}}

\begin{document}
\renewcommand{\ib}{\mathbf}
\renewcommand{\proofname}{Proof}
\renewcommand{\phi}{\varphi}
\makeatletter \headsep 10 mm \footskip 10 mm
\renewcommand{\@evenhead}%

\title{Belt diameter of $\Pi$-zonotopes}
\author{A. Garber}

\address{Lomonosov Moscow State University, Faculty of Mechanics and Mathematics, Department of Higher Geometry and Topology. Russia, 119991, Moscow, Vorob'evy gory, 1, A-1620.}
\address{Demidov Yaroslavl State University, Laboratory of Discrete and Computational Geometry, Russia, 150000, Yaroslavl, Sovetskaya str., 14.}
\email{alexeygarber@gmail.com}
\date{\today}

\begin{abstract}
A {$\Pi$-zonotope} is a zonotope that can be obtained from permutahedron by deleting zone vectors. Any face $F$ of codimension 2 of such zonotope generates its {\it belt}, i.e. the set of all facets parallel to $F$. The {\it belt diameter} of a given zonotope $Z$ is the diameter of the graph with vertices correspondent to pairs of opposite facets and with edges connect facets in one belt.

In this paper we investigate belt diameters of $\Pi$-zonotopes. We prove that any $d$-dimensional $\Pi$-zonotope ($d\geq 3$) has belt diameter at most 3. Moreover if $d$ is not greater than 6 then its belt diameter is bounded from above by 2. Also we show that these bounds are sharp. As a consequence we show that diameter of the edge graph of dual polytope for such zonotopes is not greater than 4 and 3 respectively.
\end{abstract}

\maketitle

\section{Zonotopes and parallelohedra}

\begin{defin}
A $d$-dimensional polytope $P$ is called a {\it parallelohedron} if $\mathbb{R}^d$ can be tiled in parallel copies of $P.$
\end{defin}

In 1897 Minkowski proved that any paralleohedron $P$ is centrally symmetric, has centrally symmetric facets, and projection of $P$  along any its face of codimension 2 is a two-dimensional parallelohedron \cite{Min}. Later Venkov \cite{Ven} proved that these three properties are sufficient for polytope to be a parallelohedron. The last property of Minkowski allows us to introduce a notion of {\it belt} as a set of (4 or 6) facets parallel to a given face of codimension 2. The facets of a belt are projected exactly into edges of two-dimensional parallelohedron, i.e. parallelogram or centrally symmetric hexagon.

Using a notion of belt we can introduce the belt diameter of a given parallelohedron $P$. We can construct the {\it Venkov graph} \cite{Ord} of $P$ in the following way. The vertex set of the Venkov graph is the set of pairs of opposite facets and two pairs of facets are connected with an edge if and only if there exist a belt containing both pairs. Then the {\it belt diameter} of parallelohedron $P$ is the diameter of its Venkov graph, in the other words the belt diameter of a parallelohedron is the maximal number of belts that we need to use in order to travel from one facet to another. A {\it belt path} of a parallelohedron is a sequence of its facets such that any two correspondent facets are in the same belt, so belt diameter of a polytope is the maximal length of the shortest belt path between two of its facets. Since every belt of a parallelohedron consist of 4 or 6 facets then the Venkov graph can be obtained by gluing vertices correspondent to opposite facets in the edge graph of the dual polytope.

One of the main conjecture in the parallelohedra theory is the Voronoi conjecture \cite{Vor} that claims that every parallelohedron is an affine image of Dirichlet-Voronoi polytope for some lattice. This conjecture was proved for several classes of parallelohedra in works of Voronoi \cite{Vor}, Zhitomirskii \cite{Zhit}, Erdahl \cite{Erd}, and Ordine \cite{Ord}. One of the main methods that was used in these work (except \cite{Erd}) is the method of canonical scaling introduced by Voronoi in \cite{Vor}. This method constructs a special function (canonical scaling) on the set of facets of parallelohedron by moving from one facet to another facet in the same belt. So relatively small belt diameter of parallelohedron $P$ can give us a way to prove the Voronoi conjecture for $P$ because it will be easier to prove the existence of canonical scaling for $P$. In this paper we will investigate belt diameters of one class of parallelohedra that described below.

\begin{defin}
A $d$-dimensional polytope $P$ is called a {\it zonotope} if $P$ can be represented as a projection of some cube $C^n$ of dimension $n\geq d.$ Equivalently, every zonotope can be represented as a Minkowski sum of finite number of segments. These segments can be written in the form $[\ib 0,\ib v_i], i=1,\ldots,n$ for {\it zone vectors} $\ib v_i.$ The zonotope with set of zone vectors $V=\{\ib v_1,\ldots, \ib v_n\}$ we will denote $Z(V).$
\end{defin}

There were established several results on parallelohedral zonotopes or space filling zonotopes. In 1974 P.~McMullen \cite{McM2} proved the necessary and sufficient condition for set of zone vectors $V$ to generate a space filling zonotope $Z(V).$

\begin{thm}[McMullen, 1974, \cite{McM2}]\label{2or3}
The $d$-dimensional zonotope $Z(V)$ is a parallelohedron if and only if projection of $V$ along any its $(d-2)$-dimensional subset consist of vectors of $2$ or $3$ directions.
\end{thm}

In 1999 R.~Erdahl \cite{Erd} proved the Voronoi conjecture for space filling zonotopes.

\begin{thm}[Erdahl, 1999, \cite{Erd}]
For any space filling zonotope $Z$ there exist an affine transformation $\mathcal{A}$ such that the zonotope $\mathcal{A}Z$ is the Dirichlet-Voronoi polytope for some lattice $\Lambda.$
\end{thm}

Later this result was reproved by M.~Deza and V.~Grishukhin using oriented matroids \cite{DG}.

One of the most famous zonotopes that is also a parallelohedron is a permutahedron.

\begin{defin}
The {\it permutahedron} $\Pi_d$ is the convex hull of $(d+1)!$ points in $\mathbb{R}^{d+1}$ whose coordinates are all permutations of numbers $1,2,\ldots, d+1.$ It is a zonotope with $\dfrac{d(d+1)}{2}$ zone vectors $\ib e_{ij}=\ib e_i-\ib e_j$ for $i<j$ where $\ib e_k$ is the $k$-th vector of the standard basis in $\mathbb{R}^{d+1}.$ It is easy to see that $\Pi_d$ is a $d$-dimensional polytope since all its vertices belongs to the hyperplane $x_1+\ldots+x_{d+1}=\frac{(d+1)(d+2)}{2}$ in $\mathbb{R}^{d+1}.$ Also $\Pi_d$ is a parallelohedron \cite{GP}.
\end{defin}

The set of zone vectors of $d$-dimensional permutahedron we will denote as $V(d).$

\begin{defin}
If the set of zone vectors of the zonotope $Z(V)$ is a subset of the set $V(d)$ then we will call $Z$ a {\it $\Pi$-zonotope}. 

It is easy to see that any such zonotope is parallelohedron because conditions of the theorem \ref{2or3} holds if we remove some zone vectors from the set $V$.
\end{defin}

In this paper we will investigate belt diameters of $\Pi$-zonotopes.

We will consider only the case $d\geq 3$ because case $d\leq 2$ is trivial.

\section{Graph representation of $\Pi$-zonotopes}

\begin{defin}
Given a $\Pi$-zonotope $Z=Z(V), V\subset V(d)$. A graph with $d+1$ vertices and the edge set $E$ is called a {\it graph of $Z(V)$} if there is such an enumeration of vertices of the graph by numbers $1,\ldots,d+1$ that and edge $(i,j)$ belongs to the $E$ if and only if one of the two opposite vectors $\pm(\ib e_i-\ib e_j)$ belongs to the set $V$ of zone vectors of $Z(V)$. We will denote such a graph $G_Z$.

In particular, if $Z$ is the $d$-dimensional permutahedron $\Pi_d$ then $G_{\Pi_d}$ is the complete graph with $d+1$ vertices.

And if we are given a graph $G$ with $d+1$ enumerated vertices then we can construct a $\Pi$-zonotope $Z=Z_G$ with set of zone vectors $\{\ib e_{ij}\}$ correspondent to edge set of $G.$
\end{defin}

In both constructions of a graph for a given zonotope or a zonotope for a given graph the re-enumerating of vertices of a graph does not change the metric and combinatorial properties of the zonotope since any re-enumerating of the vertices of a graph corresponds to permutation of the coordinate axis of the space $\mathbb{R}^{d+1}.$ So if graphs $G_Z$ and $G_Y$ of two $\Pi$-zonotopes are isomorphic then $Z$ and $Y$ are combinatorially equivalent.

However two combinatorially equivalent zonotopes could have different graphs. For example zonotopes of both graphs with 4 vertices on the next picture are 3-dimensional parallelepipeds because they are Minkowski sums of three linearly independent vectors. Namely vectors $\ib e_2-\ib e_1$, $\ib e_3-\ib e_2$ and $\ib e_4-\ib e_3$ for the left graph and $\ib e_2-\ib e_1$, $\ib e_3-\ib e_2$ and $\ib e_4-\ib e_2$ for the right one.

\begin{figure}[!ht]
\begin{center}
\includegraphics[height=2cm]{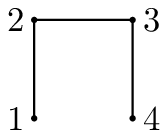}
\hskip 2cm
\includegraphics[height=2cm]{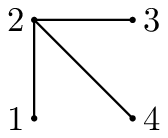}
\caption{Two different graphs correspondent to 3-dimensional parallelepiped.}
\end{center}
\end{figure}

In this section we will prove several combinatorial properties of $\Pi$-zonotopes using their representation as graphs.

\begin{lem}\label{dimgraph}
For a given graph $G$ with $d+1$ vertices and $k$ connected components the dimension of the correspondent $\Pi$-zonotope is equal to $d+1-k.$
\end{lem}

\begin{proof}
We will use an induction on the number of edges of the graph $G.$ If $G$ is an empty graph with $d+1$ vertices then the zonotope $Z_G$ is a point ($0$-dimensional) and $G$ has $d+1$ connected components and the equality $\dim Z_G=d+1-k$ holds. Every time when we add one edge $(i,j)$ there are two possibilities.
\renewcommand{\theenumi}{{\rm \roman{enumi}}}
\begin{enumerate}
\item The new edge is between vertices of one connected component. Then this new edge $(i,j)$ belongs to some cycle in this component and hence the correspondent zone vector $\ib e_{ij}$ is a sum of zone vectors correspondent to other vectors of this cycle. Thus adding the vector $\ib e_{ij}$ to the set of the zone vectors will not change the dimension of the resulting zonotope as well as the number of the connected components will not change.

\item The new edge connects two vertices from different connected components $A$ and $B.$ The dimension of the resulting zonotope can not decrease and can not increase for more than 1. We just need to show that this dimension will not remain the same. Assume that $\ib e_{ij}$ is a linear combination of zone vectors correspondent to other edges of $G.$ So it can be represented as a sum $\ib e_{ij}=\ib e_A+\ib e_B+\ib f$ where $\ib e_A$ and $\ib e_B$ are linear vectors correspondent to edges in components $A$ and $B$ respectively and $\ib f$ is a linear combination of all other vectors. But sum of all coordinates on the places correspondent to vertices in the component $A$ on the left hand side is equal to $\pm 1$ and on the right had side this sum is equal to 0 because it is zero for every vector there, we got a contradiction.
\end{enumerate}
\end{proof}

\begin{lem}\label{proj} Given a $\Pi$-zonotope with graph $G$. Its projection along vector $\ib e_{ij}$ is combintorially equivalent to a $\Pi$-zonotope whose graph $G'$ is obtained from $G$ by gluing vertices $i$ and $j$.
\end{lem}

\begin{proof}
As the first step we remark that it is does not matter on which hyperplane transversal to $\ib e_{ij}$ we projecting because different projections are equivalent under suitable affine transformation, so if initial graph has $k+1$ vertices or the same $Z\subset \mathbb{R}^{k+1}$ we will consider the projection on the plane $\displaystyle x_1+\ldots +x_{k+1}-x_j=0.$ Then any vector $\ib e_{ab}$ will project into
\begin{itemize}
\item $\ib e_{ab}$ if neither $a$ nor $b$ are not equal to $j$;
\item $\ib e_{ib}$ if $a=j$ and $b\neq i$;
\item $\ib e_{ai}$ if $b=j$ and $a\neq i$;
\item $\ib 0$ if $\ib e_{ab}=\pm \ib e_{ij}.$
\end{itemize}
So we will have a set of vectors correspondent to edges of subgraph with $k$ vertices (all beside $j$) but some edges can appear twice if there were both edges $(i,a)$ and $(j,a).$ But if we multiple some zone vector by a non-zero constant (in our case this constant equal to $\frac12$) we will not change the combinatorial type of the zonotope. Also it is easy to see that all changes listed above corresponds to gluing of vertices $i$ and $j$ in the initial graph.
\end{proof}

\begin{cor}
Any $d$-dimensional $\Pi$-zonotope $Z$ is combinatorially equivalent to a $\Pi$-zonotope with correspondent connected graph with $d+1$ vertices obtained from $\Pi_d$.
\end{cor}

\begin{proof}
If there are two components in the graph $G_Z$ then we can choose two vertices $i$ and $j$ from different components and project $Z$ along the vector $\ib e_{ij}.$ Due to lemma \ref{proj} the number of components will decrease by 1 and the number of vertices of the graph will decrease by 1 so by lemma \ref{dimgraph} dimension of polytope will not change. So after this operation we will obtain combinatorially equivalent polytope (actually we will get a linear transformation of $Z$). We can do such operation as long as graph $G_Z$ has more than one connected component. At the end we will have connected graph with $d+1$ vertices since the dimension of initial polytope $Z$ is equal to $d$.
\end{proof}

It is well known that any $k$-face of a zonotope $Z(V)$ is again a zonotope $Z(U)$ for some $k$-dimensional subset of zone vectors $U\subset V$ and backwards for any $k$-dimensional subset $U$ of $V$ that cannot be extended with other vectors from $V$ without increasing its dimension there will be a family of parallel $k$-faces of $Z(V)$ equal to $Z(U)$. Now we will describe how to find all faces of a $\Pi$-zonotope with graph $G.$

\begin{defin}
Given graph $G$ with vertex set $A$ and edge set $E$ and a non-empty subset $A'\subset A$. We call the {\it induced subgraph} $G(A')$ the graph with vertex set $A'$ and subset $E'\subset E$ of all edges of $G$ that connects two vertices from $A'$.
\end{defin}

\begin{lem}\label{pi-face}
Given a $\Pi$-zonotope $Z$ with connected graph $G.$ Any face $F$ of codimension $k$ of $Z$ determines a partition of the vertex set of the graph $G$ into $k+1$ non-empty subsets $\mathcal{A}=\{A_0,\ldots,A_k\}$ such that any induced subgraph $G(A_i)$ is connected and in that case $F=Z_{G(\mathcal{A})}$ where $G(\mathcal{A})$ denotes graph $G(A_0)\cup\ldots\cup G(A_k)$. And backwards any partition $\mathcal{A}=\{A_0,\ldots,A_k\}$ with connected induced subgraphs $G(A_i)$ determines a family of parallel faces of $Z$ of codimension $k$ that are equal to $Z_{G(\mathcal A)}.$
 \end{lem}

\begin{rem}
Compare with \cite[Example 0.10]{Zieg} and the combinatorial description of faces of permutahedron, i.e. the $\Pi$-zonotope with complete graph $G.$ It is easy to see that in the case of permutahedron any (non-ordered) partition of the vertex set of the graph will give us connected subgraphs. In this lemma different faces from one family differs by permutations of sets of partition from description of \cite[Example 0.10]{Zieg}.
\end{rem}

\begin{proof}
Consider a codimension $k$ face $F$ of the $d$-dimensional $\Pi$-zonotope $Z$ with graph $G$ with $d+1$ vertices. Let's draw a new graph $G'$ with edges correspondent to zone vectors of face $F.$ Since $F$ has dimension $d-k$ then $G'$ has $k+1$ connected components, denote vertex set of these components as $A_0,\ldots,A_k.$ It is enough to show that $G'(A_i)=G(A_i).$ Since vertices of $A_i$ are connected by edges of $G'$ then any vector $\ib e_{ab}$ with $a,b\in A_i$ can be represented as a linear combination of some zone vectors from $F$ and then any vector correspondent to some edge of $G(A_i)$ is parallel to $F$ and then $G(A_i)$ is a subgraph of $G'(A_i).$ On the other hand $G'$ is a subgraph of $G$ so $G'=G(A_0)\cup\ldots\cup G(A_k)$ and $F=Z_{G(\mathcal A)}$.

To prove the second statement of this lemma consider an arbitrary partition $\mathcal A=\{A_0, \ldots, A_k\}$ of the vertex set of $G$ such that $G(A_i)$ is connected. Then all vectors correspondent to edges of subgraph $G(\mathcal A)=G(A_0)\cup\ldots\cup G(A_k)$ forms a vector set $T$ of codimension $k.$ Consider a family of hyperplanes $\pi$ in the space of $Z$ such that every hyperplane from $\pi$ is parallel to any vector from $T$ and is not parallel to any other zone vector of $Z.$ Then supporting hyperplane of $Z$ parallel to some plane from $\pi$ determines a face of $Z$ and this face is equal to $Z(T)=Z_{G(\mathcal A)}$ as desired. And any face equal and parallel to $Z_{G(\mathcal A)}$ can be obtained in this way because there is a supporting hyperplane correspondent to this face and this plane is parallel to some face from $\pi$.
\end{proof}

\begin{lem}\label{subface}
For two partitions $\mathcal A_1$ and $\mathcal A_2$ of the vertex set of connected graph $G$ there are two incident faces $F_1$ and $F_2$ of the $\Pi$-zonotope $Z_G$ correspondent to these partitions if and only if one partition is subpartition of another and every set $X$ from any partition $\mathcal{A}_1$ or $\mathcal{A}_2$ induces connected subgraph $G(X)$.
\end{lem}

\begin{proof}
The connectedness of all induced subgraphs immediately follows from the lemma \ref{pi-face}. Consider two faces $F_1$ and $F_2$ of $Z_G$ such that $F_1$ is a face of $F_2.$ There are two partitions $\mathcal{A}_1$ and $\mathcal{A}_2$ of the vertex set of $G$ that correspondent to these faces. Since $F_1\subset F_2$ then any zone vector of $F_1$ is also a zone vector of $F_2$ and then $G(\mathcal A_1)$ is a subgraph of $G(\mathcal A_2)$. Therefore any connected component $G(\mathcal A_1)$ (i.e. set from partition) is a subcomponent of $G(\mathcal A_2)$ and the ``only if'' part is proved.

On the other hand, if we have a subpartition $\mathcal A_1$ of partition $\mathcal A_2$ then any edge of $G(\mathcal A_1)$ (a zone vector of $F_1$) is also an edge of $G(\mathcal A_2)$ (a zone vector of $F_2$). Also any zone vector of $F_2$ that is a linear combination of other zone vectors of $F_1$ is a zone vector of $F_1$ and that finishes the proof.
\end{proof}

\begin{lem}\label{belt}
Given a $\Pi$-zonotope $Z$ with connected graph $G$ and two facets $F$ and $H$ of $Z$ determined by partition of the vertex sets of $G$ into subsets $X_1, X_2$ and $Y_1, Y_2$ respectively. If facets $F$ and $H$ are in the same belt of $Z$ then one of sets $X_1$ or $X_2$ is a subset of $Y_1$ or $Y_2$. And backwards if all sets $X_i\cap Y_j$ induces connected subgraphs and one of these sets is empty then correspondent facets are in one belt.
\end{lem}

\begin{proof}
Assume that facets $F$ and $H$ are in the same belt of $Z$ then due to lemma \ref{subface} for partitions $\{X_1,X_2\}$ and $\{Y_1,Y_2\}$ must exist a join subpartition into three sets. Then this subpartition must contain four sets $X_i\cap Y_j$ and this is possible only if one of these sets is empty. Without loss of generality we can assume that $X_1\cap Y_1=\emptyset$ and then $X_1$ is a subset of $Y_2.$

On the other hand if all 4 sets $X_i\cap Y_j$ induces connected subgraphs and one of these sets is empty then this partition generates a family of faces of codimension 2 of $Z$. And some faces from this family due to lemma \ref{subface} belongs to both pairs of opposite facets generated by partitions $\{X_1,X_2\}$ and $\{Y_1,Y_2\}.$
\end{proof}

\section{Symmetric $\Pi$-zonotopes and their representation}

\begin{defin}\label{symdef}
Consider two $(d-1)$-dimensional vector sets $V_1$ and $V_2$ of $d-1$ vectors each in $\mathbb{R}^d$. We will call these two sets {\it conjugate} if for any vectors $\ib e_1\in V_1$ and $\ib e_2\in V_2$ we have $\dim \{\ib e_1\cup V_2\}=\dim \{e_2\cup V_1\}=d.$ The zonotope $Z(V_1\cup V_2)$ is called {\it symmetric} zonotope.
\end{defin}

The notion of symmetric zonotopes is useful for finding maximal belt diameters of $d$-dimensional zonotopes. This result is proved in \cite[Cor. 4.3]{Gar}.

\begin{lem}[\cite{Gar}]
If $\xi(d)$ is the maximal belt diameter of $d$-dimensional space filling symmetric zonotope then belt diameter of any $d$-dimensional space filling zonotope is not greater than $\max\limits_{1\leq i\leq d}\xi(i)$.
\end{lem}

Here we will point how to proof the same lemma for $\Pi$-zonotopes.

\begin{lem}\label{max}
If $\xi_\Pi(d)$ is the maximal belt diameter of $d$-dimensional symmetric $\Pi$-zonotopes then belt diameter of any $d$-dimensional $\Pi$-zonotope is not greater than $\max\limits_{1\leq i\leq d}\xi_\Pi(i)$. Moreover this maximum will be achieved on conjugated facets of symmetric $\Pi$-zonotope.
{\sloppy

}
\end{lem}

\begin{rem} If we decline any restrictions in this or previous lemma then the statement can be formulated as follows. The maximal belt diameter of $d$-dimensional zonotope is not greater $\max\limits_{1\leq i\leq d}\xi'(i)$ where $\xi'(i)$ denotes the maximal belt diameter of $i$-dimensional symmetric zonotope. 

But this result can be obtained straightforward since for any two facets $Z(E)$ and $Z(F)$ of zonotope $Z(V)$ we can find belt path of length at most $d-1$ between them. To do that assume that set $E$ contains linearly independent vectors $\ib e_i, 1\leq i\leq d-1$ and $F$ contains linearly independent vectors $\ib f_i, 1\leq i\leq d-1$, also these sets can contain other vectors too. Then all $(d-1)$-dimensional sets $EF_i=\{\ib f_1,\ldots, \ib f_i,\ib e_{i+1},\ldots, \ib e_{d-1}\}$ gives us a belt path between facets $Z(E)$ and $Z(F)$ of length at most $d-1$ since $EF_0=E$ and $EF_{d-1}=F$ and each time when we get a $(d-1)$-dimensional set $EF_i$ we make a step using exactly one belt from the previuos one.

So we got an upper bound for maximal belt diameter of arbitrary $d$-dimensional zonotope and an example of $d$-dimensional zonotope with belt diameter $d-1$ can be obtained from two $(d-1)$-dimensional sets $E$ and $F$ in general position, i.e. we can take any $2d-2$ vectors such that any $d$ of them are linearly independent (so $\xi'(i)=i-1$).
\end{rem}

\begin{proof}
The idea of proof is to show that for any two facets $F_1$ and $F_2$ of a given $\Pi$-zonotope $Z$ with belt distance $k$ there exist a symmetric $\Pi$-zonotope $Z(V_1\cup V_2)$ of dimension at most $d$ with belt distance between facets correspondent to sets $V_1$ and $V_2$ at least $k.$

If facets $F_1$ and $F_2$ has a common zone vector $\ib e$ then we can project $Z$ along $\ib e$ and this will not decrease belt distance, i.e. belt distance between projections of $F_1$ and $F_2$ will be not smaller than belt distance between initial faces because any belt path on the resulted zonotope is a projection of belt path on initial zonotope of the same length. Also after projecting we will obtain a $\Pi$-zonotope as we show in lemma \ref{proj}.

Otherwise we can remove any zone vector from the face $F_1$ until we can do it without decreasing its dimension. And this operation again will not decrease the belt distance between $F_1$ and $F_2$ since for any belt path in the resulting zonotope will be a belt path in the initial zonotope with the same generating vector sets and the zonotope $Z$ will remain a $\Pi$-zonotope. The same operation we can do with another facet $F_2$ or with any zone vector not from $F_1$ and $F_2$ until both facets will contain exactly $d-1$ linearly independent vectors and there will be no other zone vectors in $Z.$ In that case sets of zone vectors of $F_1$ and $F_2$ are conjugated and $Z$ is a symmetric zonotope $Z(F_1\cup F_2)$.
\end{proof}

Now we will describe several properties of graph representation of $d$-dimensional symmetric $\Pi$-zonotopes with connected graphs on $d+1$ vertices. Consider a symmetric $\Pi$-zonotope $Z(V_R\cup V_B)$ and its graph $G.$ Edges of $G$ can be colored in red or blue color whether they correspondent to zone vectors of $V_R$ or $V_B$ respectively, we will call these facets red and blue. Also we call correspondent red and blue subgraphs of $G_Z$ on $d+1$ vertices $G_R(Z)$ and $G_B(Z)$ respectively.

In the following text we will use several well-known notions from graph theory like {\it tree}, {\it forest} and {\it leaf}. Detailed definitions and properties can be found in \cite{Dies}.

It is well known that every tree with $n$ vertices has $n-1$ edges and every forest with $k$ trees and $n$ vertices has $n-k$ edges. Moreover if connected graph with $n$ vertices has $n-1$ edges then it is a tree and if graph with $n$ vertices and $k$ connected components has $n-k$ edges then it is a forest with $k$ trees.

\begin{lem}\label{trees}
Both graphs $G_R(Z)$ and $G_B(Z)$ are forests with two components and  for any blue edge $\ib e_b$ and any red edge $\ib e_r$ graphs $G_R\cup\{\ib e_b\}$ and $G_B\cup\{\ib e_r\}$ are trees, so every blue edge connects two different red components and every red edge connects different blue components.
\end{lem}

\begin{proof}
It is enough to note that blue graph $G_B$ has $d-1$ edges and $d+1$ vertices and since it is correspondent to $(d-1)$-dimensional set of zone vectors $V_B$ it has two connected components, so $G_B$ is a forest with two trees and by the same reason $G_R$ is also a forest with two trees. Moreover, due to definition \ref{symdef} if we add a red edge $\ib e_r$ to blue subgraph then it will determine $d$-dimensional set of zone vectors so this graph $G_B\cup\{\ib e_r\}$ must have only one connected component, thus it must be a tree and $\ib e_r$ must connect two different blue components. The same is true for every blue edge.
\end{proof}

\begin{cor}\label{oneedge}
There are no cycles in $G_Z$ that contain exactly one red or exactly one blue edge.
\end{cor}

\begin{cor}\label{even}
Two vertices in one red (blue) component are in the same blue (red) component if and only the distance between these vertices in the red (blue) subgraph is even. Therefore the graph $G_Z$ is bipartite.
\end{cor}

\begin{proof}
We need only to mention that due to lemma \ref{trees} if we go from one vertex to another by red edge then we change a blue component and in a tree there is a unique path between any two vertices.

To proove the second assertion of this corollary it is enough to mention that for any cycle numbers of blue and red edges in it has the same parity since any blue edgse changes red component and any red edge changes blue component and in the end of the cycle we need to come to the same blue and the same red component. Hence any cycle in $G_Z$ has even length.
\end{proof}

We will denote sets of vertices of red and blue trees as $R_1$, $R_2$ and $B_1$, $B_2$ respectively.

\begin{lem}\label{sinvert}
If one of red components $R_1$ is a single vertex $r$ then blue subgraph also contains an isolated vertex $b$ and all red edges connect vertex $b$ with all other $d-1$ vertices of $G_Z$ except $r$ and all blue edges connects $r$ with all other vertices except $b,$ i.e. the graph $G_Z$ is the complete bipartite graph $K_{2,d-1}$.
\end{lem}

\begin{proof}
Since every blue edge connects two different red component then any of $d-1$ blue edges has $r$ as a vertex. Then blue tree that contains $r$ also contains $d-1$ other vertices of $G_Z$ and another blue component has exactly one vertex, denote isolated vertex of blue subgraph as $b$. By the same reason any red edge is incident to the vertex $b$. But there is no red edge $rb$ since $r$ is an isolated vertex in the red subgraph and then all red edges connects $b$ with all vertices except $r$ as it is shown on the following picture.
\begin{figure}[!ht]
\begin{center}
\includegraphics[height=3cm]{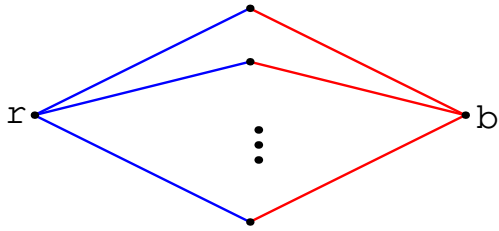}
\caption{Graph of symmetric $\Pi$-zonotope with isolated vertex as red component.}
\end{center}
\end{figure}
\end{proof}

\section{Belt diameters of $\Pi$-zonotopes}

\begin{lem}\label{sinvertdist}
Under condition of the lemma $\ref{sinvert}$ the belt distance between red and blue facets of the zonotope $Z$ is equal to $2$.
\end{lem}

\begin{proof}
Consider a facet $F$ of $Z$ determined by partition of the vertex set of $G_Z$ into two sets. The first set contains $d$ vertices including $r$ and $b$ and the second set contains the single remaining vertex. By lemma \ref{belt} this facet $F$ is in one belt with facet $V_R$ and in one belt with facet $V_B$ and correspondent belt distance between $V_B$ and $V_R$ is not greater than 2. We showed correspondent facets (partitions into two sets) on the next picture, correspondent partitions illustrated by dashed lines.

\begin{figure}[!ht]
\begin{center}
\parbox[c]{4.72cm}{\includegraphics[height=2.5cm]{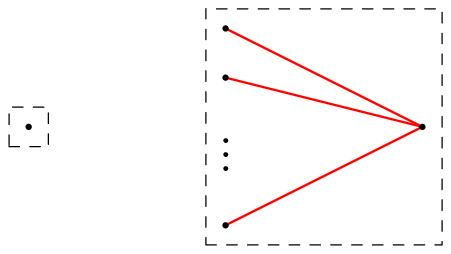}}$\longleftrightarrow$
\parbox[c]{4.72cm}{\includegraphics[height=2.5cm]{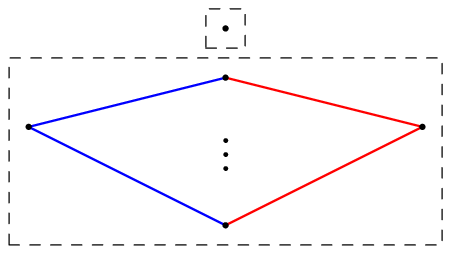}}$\longleftrightarrow$
\parbox[c]{4.72cm}{\includegraphics[height=2.5cm]{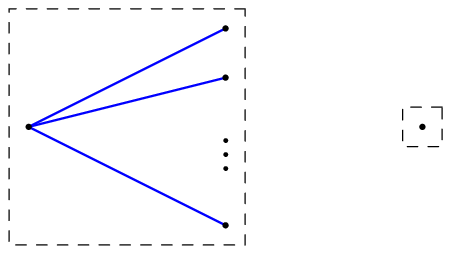}}
\caption{Belt path between facets $V_R$ on the left and $V_B$ on the right through the facet $F$.}
\end{center}
\end{figure}

We still need to show that red and blue facets are not in one belt. Assume converse and apply lemma \ref{belt} to partitions correspondent to these facets. Then three intersection sets must induce connected subgraphs but  this statement is false because these three sets are $\{r\}$, $\{b\}$ and all remaining $(d-1)$ vertices. We got a contradiction and that means that belt distance between red and blue facets is 2.
\end{proof}

\begin{lem}\label{leaves}
The belt distance between red and facets in $\Pi$-zonotope $Z(V_R\cup V_B)$ is equal to $2$ if and only if there is a vertex $A$ that is a leaf in both red and blue subgraphs $G_R$ and $G_B.$
\end{lem}

\begin{proof}
If one red or blue component consist of one vertex then we already proved the statement in lemmas \ref{sinvert} and \ref{sinvertdist}. So from now we will consider only the case of at least two vertices in every red or blue connected component.

If $A$ is a leaf in both red and blue subgraphs then consider a partition of vertex set of graph $G_Z$ into two sets one of which consist of the single vertex $A.$ By lemma \ref{pi-face} this partition determines a facet $F$ of $Z$ because if we remove vertex $A$ from $G$ then we will get a connected graph since there are exactly two red components and at least one blue edge that connects different components due to lemma \ref{trees}. And by lemma \ref{belt} this facet $F$ is in one belt with each red or blue facets because removing a leaf from a tree will give us a new connected graph.

Now assume that belt distance between red and blue facets is equal to 2, then there exists a facet $F$ with correspondent partition of vertex set of graph $G$ into sets $F_1$ and $F_2$ that satisfy lemma \ref{belt} for both red and blue facets. So one of these sets is a subset of one of red vertex sets $R_1$ or $R_2.$ Without loss of generality we can assume that $F_1\subset R_1.$ If $|F_1|>1$ (here and later $|X|$ denotes the cardinality of the vertex set $X$) then both subgraphs $G(F_2)$ and $G(F_1)$ has at least one red edge because $F_2\supset R_2$ and $R_2$ has at least two vertices and $R_2$ is connected by red edges, moreover by lemma \ref{belt} the intersection set $R_1\cap F_1=F_1$ must be connected and there are only red edges in subgraph $G(R_1)$. By lemma \ref{belt} applied to the blue facet and the facet $F$ one of sets $F_1$ and $F_2$ is a subset of one of sets $B_1$ and $B_2$ but in this case there is a red edge that connects two vertices from $B_1$ or two vertices from $B_2$ and this contradicts with lemma \ref{trees}, so $|F_1|=1$ and $F_1=\{A\}$. Again by lemma \ref{belt} all three induced subgraphs $G(\{A\}), G(R_1\setminus\{A\})$ and $G(R_2)$ must be connected and graph $G(R_1\setminus\{A\})$ is connected if and only if $A$ is a leaf of the red subgraph. By the same reason $A$ must be a leaf of the blue subgraph of $G.$
\end{proof}

\begin{thm}\label{pimax}
The belt distance between red and blue facets in a $d$-dimensional symmetric $\Pi$-zonotope is not greater than $2$ if $3\leq d\leq 6$ or $d=8$ and is not greater than $3$ in other cases. These bounds are sharp. So in notations of lemma $\ref{max}$ $$\xi_\Pi(d)=\left\{
\begin{array}{ll}
2, & \text{ if }3\leq d\leq 6 \text { or } d=8,\\
3, &\text{ if }d= 7 \text { or } d\geq 9.
\end{array}
\right.$$
\end{thm}

\begin{proof}
Note that zonotope from lemma \ref{sinvert} satisfies this lemma and further we will consider that all red and blue components has at least two vertices. Also this zonotope gives us an example for $d\leq 6$ and $d=8$ so we will need to construct examples only for other dimensions. Now suppose that the zonotope $Z(V_R\cup V_B)$ is distinct from the zonotope described in lemma \ref{sinvert}.

We have three possibilities. If $d\leq 6$ then every red component is a tree with at least $2$ vertices and then there are at least two leaves in each component so there are at least four red leaves and by the same reason there are at least four blue leaves. There are $d+1\leq 7$ vertices in graph $G$ in total so there is a vertex that is a blue and a red leaf at the same time. So we can apply lemma \ref{leaves} and this case is done.

If $d=8$ then there are $9$ vertices in graph $G$ and if each of red and blue forests has at least $5$ leaves then this case is similar to the previous one. Assume that there is no vertex of $G$ that is red and blue leaf at the same time then one of forests, say red, has exactly $4$ leaves and in this case both trees $G(R_1)$ and $G(R_2)$ must be just some simple paths. 

There are three possibilities for number of vertices in sets $R_1$ and $R_2$: 2 and 7, 3 and 6, 4 and 5. In all three cases one can easily check by simple enumeration that any graph that satisfies lemma \ref{trees} and corollary \ref{oneedge} has a vertex of degree 2 that incident to one red and one blue edges.

And the only case that we still need to prove is $d=7$ or $d\geq 9.$ In graph $G$ there are $d+1$ vertices and $2d-2$ edges so there is a vertex $A$ with degree at most 3. If $A$ has degree 2 then it is a leaf in both red and blue subgraph and belt distance between red and blue facets is 2. If $A$ has degree 3 then we can assume that there is one red edge and two blue edges at $A.$ Consider the partition of the vertex set $V$ of $G$ into sets $\{A\}$ and $V\setminus\{A\}.$ This partition determines a facet of $Z$ because red edges between vertices of $V\setminus\{A\}$ forms two trees and there are at least one blue edge in correspondent induced subgraph since there are at least 6 blue edges in $G$ and we deleted only 2 of them. By lemma \ref{belt} facet $F$ is in one belt with the red facet of $Z.$

The facet $F$ and the blue facet of $Z$ has $d-3$ joint linearly independent zone vectors (all except two blue edges at $A$) and if we project $Z$ along these $d-3$ vectors we will not decrease belt distance between $F$ and the blue facet. After this projection we will obtain a $3$-dimensional $\Pi$-zonotope and its belt diameter is at most 2 so belt distance in $Z$ between the red and the blue facets was at most 3.

Now for $d=7$ and $d\geq 9$ we will construct examples of graphs of symmetric $\Pi$-zonotopes with no vertices that is red and blue leaves at the same time. If these graph will satisfy lemma \ref{trees} then we will show that for every such $d$ there exist a symmetric $d$-dimensional $\Pi$-zonotope with belt distance 3 between red and blue facets.

For odd dimension $d=2n+3$ with $n\geq 2$ the graph will have $2n+4$ vertices, denote all vertices as shown on the following figure. Then our graph will have following blue edges: $A_1B_{2j}$ for all $j$ from 1 to $n$, $A_3B_{2n}$, $A_2B_1$, $A_4B_{2j-1}$ for all $j$ from 1 to $n.$ It is easy to check conditions of lemma \ref{trees} for this graph.

\begin{figure}[!ht]
\begin{center}
\includegraphics[height=3cm]{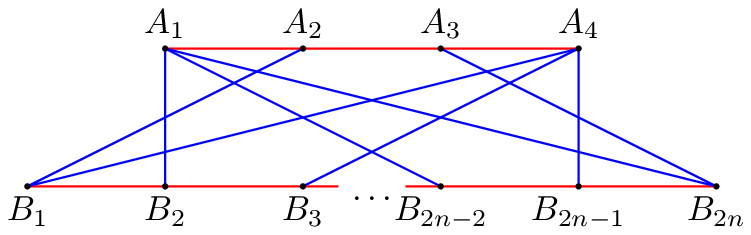}
\caption{Example of graph for odd-dimension zonotope.}
\end{center}
\end{figure}

For even dimension $d=2n+4$ with $n\geq 3$ the graph will have $2n+5$ vertices, denote all vertices as shown on the following figure. Then our graph will have following blue edges: $A_1B_1$, $A_1B_3$, $A_3B_1$, $A_5B_{2j-1}$ for all $j$ from 2 to $n$, $A_2B_{2j}$ for all $j$ from 1 to $n$, $A_4B_{2n}.$ It is easy to check conditions of lemma \ref{trees} for this graph.

\begin{figure}[!ht]
\begin{center}
\includegraphics[height=3cm]{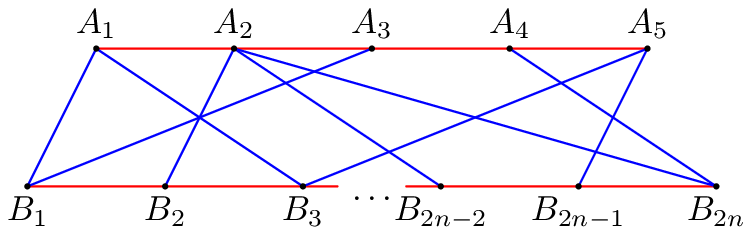}
\caption{Example of graph for odd-dimension zonotope.}
\end{center}
\end{figure}
\end{proof}

\begin{thm}
The maximal belt diameter of $d$-dimensional $\Pi$-zonotope is not greater than $2$ if $3\leq d\leq 6$ and is not greater than $3$ if $d\geq 7,$ these bounds are sharp for any dimension.
\end{thm}

\begin{proof}
For all dimensions except $d=8$ this theorem immediately follows from lemma \ref{max} and theorem \ref{pimax}. For $d=8$ we have an estimate $3$ for maximal belt diameter from lemma \ref{max} and theorem \ref{pimax} and we need to find an example of (non-symmetric) $\Pi$-zonotope with belt diameter 3. Consider a zonotope with the graph on the following figure and two facets $F_1$ and $F_2$ of this zonotope determined by partitions $X_1=\{A_1A_2A_3A_4A_5\},Y_1=\{B_1B_2B_3B_4\}$ and $X_2=\{A_1A_2A_4B_2B_4\},Y_2=\{A_3A_5B_1B_3\}$ respectively.

\begin{figure}[!ht]
\begin{center}
\includegraphics[height=3cm]{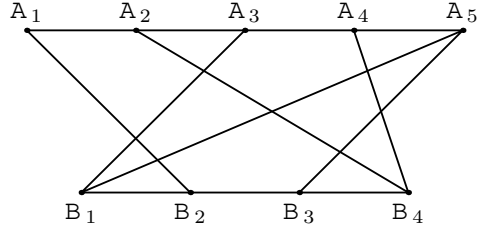}
\caption{Graph of $8$-dimensional $\Pi$-zonotope with belt diameter 3.}
\end{center}
\end{figure}

We need to show that there is no facet $F$ with partition into sets $X$ and $Y$ of the vertex set of this graph that satisfy lemma \ref{pi-face} for both facets $F_1$ and $F_2.$ Assume that such a face exists. Then by lemma \ref{belt} one of sets $X$ or $Y$ must contain one of sets $X_1,Y_1$ and $X$ or $Y$ must contain $X_2$ or $Y_2.$ But all four intersections $X_i\cap Y_j$ are nonempty and then one of sets, say $X$, must contain one subset from $X_1,Y_1$ and one from $X_2,Y_2.$

We have four possibilities:

\begin{itemize}
\item The set $X$ contains $X_1$ and $X_2.$ Then $X$ contains points $A_1,A_2,A_3,A_4,A_5,B_2,B_4.$ Due to lemma \ref{belt} intersection $X\cap Y_2$ must induce a connected subgraph and then $B_1\in X.$ On the other hand $X\cap Y_1$ must induce a connected subgraph so $B_3\in X$ and $Y$ is empty and this impossible.

\item The set $X$ contains $X_1$ and $Y_2.$ Then $X$ contains points $A_1,A_2,A_3,A_4,A_5,B_1,B_3.$ Due to lemma \ref{belt} intersection $X\cap Y_1$ must induce a connected subgraph and then $B_4\in X.$ On the other hand $X\cap X_2$ must induce a connected subgraph so $B_2\in X$ and $Y$ is empty and this impossible.

\item The set $X$ contains $Y_1$ and $X_2.$ Then $X$ contains points $A_1,A_2,A_4,B_1,B_2,B_3,B_4.$ Due to lemma \ref{belt} intersection $X\cap X_1$ must induce a connected subgraph and then $A_3\in X.$ On the other hand $X\cap Y_2$ must induce a connected subgraph so $A_5\in X$ and $Y$ is empty and this impossible.

\item The set $X$ contains $Y_1$ and $Y_2.$ Then $X$ contains points $A_3,A_5,B_1,B_2,B_3,B_4.$ Due to lemma \ref{belt} intersection $X\cap X_1$ must induce a connected subgraph and then $A_4\in X.$ On the other hand $X\cap X_2$ must induce a connected subgraph so $A_1$ and $A_2$ are in $X$ and $Y$ is empty and this impossible.
\end{itemize}

So all four possibilities are impossible and the theorem is proved.
\end{proof}

Now we will establish connection between belt diameter of parallelohedron and its combinatorial diameter.

\begin{defin}
The {\it combinatorial diameter} of polytope $P$ is the diameter of edge graph of its dual polytope, i.e. graph with vertices correspondent to facets of $P$ and with edges connecting facets adjacent by a face of codimension 2.
\end{defin}

\begin{thm}
If $d$-dimensional parallelohedron $P$ has belt diameter $k$ then its combinatorial diameter is not greater than $k+1$.
\end{thm}

\begin{proof}
Consider two arbitrary facets of $P$ and sequence of at most $k$ belts $\Gamma_1,\ldots, \Gamma_n$ that connects these facets. If we construct a path on the edge graph of dual polytope using only facets of these belts then on each belt we will need to do $s_i$ steps where $s_i$ is equal to $1$ or $2.$ Now we will show how to decrease the number of belts with two steps.

If there are two consecutive belts $\Gamma_i$ and $\Gamma_{i+1}$ with $s_i=s_{i+1}=2$ (this automatically means that both these belts consist of 6 facets) then we replace path on each belt on the complementary half-belts consist of one step. Also if we have two consecutive belts $\Gamma_i$ and $\Gamma_{i+1}$ with $s_i=2$ and $s_{i+1}=1$ then we can flip these numbers of steps using complementary half-belts of $1$ and at most $2$ steps. The illustration of these two operations is on the next figure, red segments illustrates the path on edges of the dual polytope.

\begin{figure}[!ht]
\begin{center}
\parbox[c]{1.7cm}{\includegraphics[height=2cm]{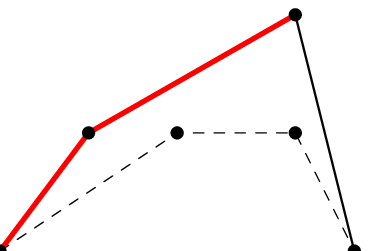}}
$\longrightarrow$
\parbox[c]{1.7cm}{\includegraphics[height=2cm]{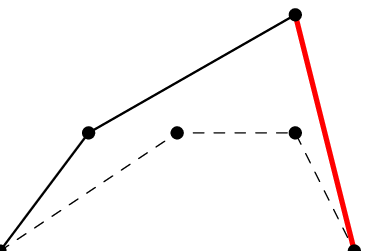}}
\hskip 3cm
\parbox[c]{1.7cm}{\includegraphics[height=2cm]{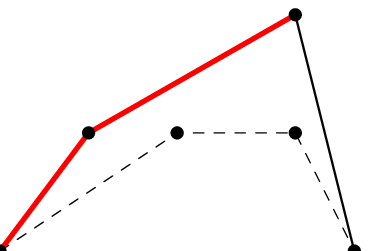}}
$\longrightarrow$
\parbox[c]{1.7cm}{\includegraphics[height=2cm]{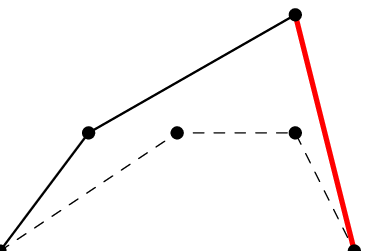}}
\caption{Two operations on belts.}
\end{center}
\end{figure}

Using these operations we can take two closest $s_i$'s that are equal to 2 and replace them by 1's. So in the end we will get the same sequence of belts but with at most one belt with 2 steps on it, so for these two facets we constructed a path of length at most $k+1.$
\end{proof}

\begin{cor}
The combinatorial diameter of $d$-dimensional $\Pi$-zonotope is at most $3$ if $d\leq 6$ and at most $4$ if $d\geq 7$.
\end{cor}

\section{Acknowledgements}

The author would like to thank the Fields Institute at Toronto, ON and Professor Robert Erdahl from Queen's University at Kingston, ON. This work was finished during visit to these places.

This work is financially supported by RFBR (projects 11-01-00633-a and 11-01-00735-a), by the Russian government project 11.G34.31.0053 and by Federal Program ``Scientific and pedagogical staff of innovative Russia''.

\end{document}